\documentclass[a4paper,11pt]{amsart}

\usepackage{amssymb, amsthm, amsmath}
\usepackage{tikz}

\usepackage{hyperref}

\usepackage{bbm}

\usepackage[hmargin=3.5cm,foot=0.7cm]{geometry}
\usepackage[utf8x]{inputenc}

\usepackage[english]{babel}
\usepackage[abbrev]{amsrefs}

\usepackage{tikz}
\usepackage{tikz-cd}
\usetikzlibrary{matrix,arrows}
\usepackage{pgfplots}
\pgfplotsset{width=7cm,compat=1.8}
\usepgfplotslibrary{polar}

\usepackage{enumitem}
\usepackage{verbatim}

\usepackage[colorinlistoftodos,prependcaption,textsize=tiny]{todonotes}
\newtheorem{theorem}{Theorem}[section]
\newtheorem*{theorem*}{Theorem}
\newtheorem{definition}[theorem]{Definition}

\theoremstyle{plain}
\newtheorem{corollary}[theorem]{Corollary}
\newtheorem*{corollary*}{Corollary}
\newtheorem{lemma}[theorem]{Lemma}
\newtheorem{proposition}[theorem]{Proposition}

\newcounter{mt}

\newtheorem{MainTheorem}[mt]{Theorem}
\newtheorem{MainCorollary}[mt]{Corollary}

\newtheorem*{lemma*}{Lemma}
\newtheorem*{question*}{Question}

\theoremstyle{definition}

\newtheorem{remark}[theorem]{Remark}

\makeatletter
\newcommand{\tpitchfork}{%
	\vbox{
		\baselineskip\z@skip
		\lineskip-.52ex
		\lineskiplimit\maxdimen
		\m@th
		\ialign{##\crcr\hidewidth\smash{$-$}\hidewidth\crcr$\pitchfork$\crcr}
	}%
}
\makeatother

\BibSpec{article}{%
	+{}{\PrintAuthors}  		{author}
	+{,}{ \textit}     		{title}
	+{,}{ }             		{journal}
	+{}{ \textbf}       		{volume}
	+{}{ \parenthesize} 		{date}
	+{}{, no. } 		{number}
	+{,}{ }      	      		{conference}
	+{,}{ }      	      		{book}
	+{,}{ }            		{pages}
	+{,}{ }            	 	{note}
	+{,}{ }            	 	{status}
	+{,}{  \texttt } {eprint}
	+{.}{}              {transition}
}

\BibSpec{book}{%
	+{}{\PrintAuthors}  {author}
	+{,}{ \textit}      {title}
	+{,}{ }             {publisher}
	+{,}{ }             {place}
	+{,}{ }             {date}
	+{.}{}              {transition}
}

\BibSpec{incollection}{%
	+{}  {\PrintAuthors}                {author}
	+{,} { \textit}                     {title}
	+{.} { }                            {part}
	+{:} { \textit}                     {subtitle}
	+{,} { \PrintContributions}         {contribution}
	+{,} { \PrintConference}            {conference}
	+{,} { }                            {booktitle}
	+{,} { pp.~}                        {pages}
	+{,}{ }       		{series}
	+{}{ \textbf}       		{volume}
	+{,} { }                            {publisher}
	+{,} { }                 {date}
	+{,} { }                            {status}
	+{,} { \PrintDOI}                   {doi}
	+{,} { available at \eprint}        {eprint}
	+{}  { \parenthesize}               {language}
	+{}  { \PrintTranslation}           {translation}
	+{;} { \PrintReprint}               {reprint}
	+{.} { }                            {note}
	+{.} {}                             {transition}
}

\newcommand{\CC}{\mathbb{C}}

\newcommand{\RR}{\mathbb{R}}

\newcommand{\VConv}{\mathrm{VConv}}
\newcommand{\Conv}{\mathrm{Conv}}
\newcommand{\ConvO}{\mathrm{Conv}_{(0)}}

\DeclareMathOperator{\dom}{dom}

\newcommand{\GW}{\mathrm{GW}}

\newcommand{\MeasC}{\mathcal{M}_c}

\newcommand{\DistribC}{\mathcal{D}_c'}
\newcommand{\Distrib}{\mathcal{D}'}

\newcommand{\convexbodies}{\mathcal{K}^n}

\DeclareMathOperator{\GL}{GL}
\DeclareMathOperator{\SL}{SL}

\DeclareMathOperator{\vol}{vol}
\DeclareMathOperator{\id}{id}
\DeclareMathOperator{\supp}{supp}

\DeclareMathOperator{\Grass}{Gr}

\DeclareMathOperator{\linspan}{span}

\DeclareMathOperator{\interior}{int}

\title{Equivariant Valuations on Convex Functions}
\author{Georg C. Hofst\"atter}
\address{Institute of Discrete Mathematics and Geometry, TU Wien, 1040 Vienna, Austria}
\email{georg.hofstaetter@tuwien.ac.at}
\author{Jonas Knoerr}
\address{Institute of Discrete Mathematics and Geometry, TU Wien, 1040 Vienna, Austria}
\email{jonas.knoerr@tuwien.ac.at}

\thanks{{\it MSC classification}:
	52A41, 
	26B25, 
	52B45. 
}
 \date{\today}

\begin{document}
	
	\begin{abstract}
		We classify all continuous valuations on the space of finite convex functions with values in the same space which are dually epi-translation-invariant and equi- resp.\ contravariant with respect to volume-preserving linear maps. We thereby identify the valuation-theoretic functional analogues of the difference body map and show that there does not exist a generalization of the projection body map in this setting. This non-existence result is shown to also hold true for valuations with values in the space of convex functions that are finite in a neighborhood of the origin.
	\end{abstract}
	
	\maketitle
	
	\section{Introduction}
	Let $\convexbodies$ denote the set of all convex bodies in $\RR^n$ (compact, convex, non-empty subsets); throughout, we assume $n \geq 2$. The difference body map $\mathrm{D}: \convexbodies\to\convexbodies$ and the projection body map $\Pi:\convexbodies\to\convexbodies$ are defined by $\mathrm{D} K = K + (-K)$ and
	\begin{align*}
		\vol_1(\Pi K | \linspan\{u\}) = \vol_{n-1}(K | u^\perp), \quad u \in \RR^n \setminus \{0\}, K \in \convexbodies,
	\end{align*}
	where addition is Minkowski addition, $\vol_k$ denotes the $k$-dimensional Lebesgue measure, $\cdot|\cdot$ denotes orthogonal projection, and $(\cdot)^\perp$ denotes the orthogonal complement. Projection bodies were introduced by Minkowski at the beginning of the previous century and have since then proved essential in the study of projections of convex bodies (see, e.g., \cite{Gardner2006, Schneider2014} and the references therein). Notably, projection and difference bodies play a key role in the solution of the Rogers--Shephard problem \cite{Shephard1964, Petty1967, Schneider1967}, in central inequalities in convex geometry like Petty's projection inequality~\cite{Petty1971}, Zhang's inequality~\cite{Zhang1991} and the Rogers--Shephard~\cite{Rogers1958b} inequality, and are the key objects in Petty's conjectured inequality~\cite{Petty1971}, which is still an open problem.
	
	About twenty years ago, Ludwig~\cite{Ludwig2002b, Ludwig2005} revealed a new perspective on projection and difference bodies. In her seminal work, she characterized them by their compatibility with respect to volume-preserving linear maps and a valuation property:
	\begin{theorem}[\cite{Ludwig2002b, Ludwig2005}]\label{thm:charProjDiffBody}
		Let $\Phi: \convexbodies \to \convexbodies$ be a continuous, translation-invariant Minkowski valuation.
		\begin{itemize}
			\item If $\Phi$ is $\SL(n,\RR)$-equivariant, then $\Phi = c \mathrm{D}$ for some $c \in \RR$.
			\item If $\Phi$ is $\SL(n,\RR)$-contravariant, then $\Phi = c \Pi$ for some $c \in \RR$. 
		\end{itemize}
	\end{theorem}
	Here, a map $\Phi: \convexbodies \to \convexbodies$ is called a Minkowski valuation if  $\Phi(K \cup L) + \Phi(K \cap L) = \Phi(K) + \Phi(L)$ whenever the union $K \cup L$ of $K, L \in \convexbodies$ is convex. It is called $G$-equivariant if $\Phi(gK)=g\Phi(K)$, and $G$-contravariant if $\Phi(gK) = g^{-T}\Phi(K)$ for all $K \in \convexbodies$ and $g \in G$ in a subgroup $G$ of the general linear group $\GL(n,\RR)$.

	Following Theorem~\ref{thm:charProjDiffBody}, many fundamental geometric constructions including ($L_p$)~intersection bodies, centroid bodies and affine surface area were characterized as equivariant valuations (see \cite{Abardia2012b, Abardia2011b,Abardia2015, Haberl2006, Ludwig2006, Ludwig2003,Ludwig2010, Haberl2012b}).
	
	\medskip
	
	In recent years, substantial efforts were made to transfer and generalize geometric concepts (for convex bodies) to a functional setting, e.g., to convex or log-concave functions, see \cite{Hofstaetter2021, Hofstaetter2023, Artstein2004, Knoerr2023, Kolesnikov2020,Rotem2012, Rotem2020, Rotem2021, Alesker2019, Colesanti2017d, Colesanti2013, Colesanti2005, Colesanti2017, Colesanti2021, Colesanti2022, Haddad2020, Knoerr2023b, Mussnig2021,Mussnig2021b}. While there are several natural candidates for a difference function map (examined, e.g., in \cite{Colesanti2006}), there is so far no suitable candidate for a generalization of the projection body map. In this article, we contribute to the solution of this problem, taking the perspective of valuation theory inspired by Theorem~\ref{thm:charProjDiffBody}. More concretely put, we aim to identify the correct valuation-theoretic analogues of difference and projection bodies on the space $\Conv(\RR^n, \RR)$ of finite convex functions $f:\RR^n \to \RR$.

	In this setting, a valuation on $\Conv(\RR^n, \RR)$ is a map $\mu: \Conv(\RR^n,\RR) \to (\mathbb{A},+)$ into an abelian semi-group $(\mathbb{A},+)$ satisfying
	\begin{align*}
		\mu(\max\{f,h\}) + \mu(\min\{f,h\}) = \mu(f) + \mu(h),
	\end{align*}
	whenever $f,h,\min\{f,h\} \in \Conv(\RR^n, \RR)$, where we denote by $\max$ resp.\ $\min$ the pointwise maximum resp. minimum. Such a valuation is called \emph{dually epi-trans\-lation-invariant} if $\mu(f+\ell) = \mu(f)$ for all $f \in \Conv(\RR^n, \RR)$ and all affine maps~$\ell:\RR^n\rightarrow\RR$.  We refer to \cite{Colesanti2019b, Colesanti2019, Knoerr2020a,Knoerr2023} for the geometric interpretation of this notion and an interpretation in terms of translation invariant valuations on convex bodies. Continuity is defined with respect to the topology induced by epi-convergence of convex functions, which for $\Conv(\RR^n,\RR)$ coincides with local uniform as well as pointwise convergence. 
	
	Our first main result is a classification of $\SL(n,\RR)$-equivariant valuations on $\Conv(\RR^n, \RR)$, thereby determining the valuation-theoretic difference function maps. Here, a map $\Psi:\Conv(\RR^n,\RR)\rightarrow\Conv(\RR^n,\RR)$ is called $G$-equivariant if $\Psi(f\circ g)=\Psi(f)\circ g$ for $f\in\Conv(\RR^n,\RR)$ and $g\in G$ in a subgroup $G\subset \GL(n,\RR)$. We denote by $\MeasC^+(\RR)$ the space of non-negative, finite Borel measures on $\RR$ with compact support.
	\begin{MainTheorem}
		\label{mthm:charSLRequivVal}
		A map $\Psi: \Conv(\RR^n, \RR) \rightarrow \Conv(\RR^n,\RR)$ is a continuous, dually epi-translation-invariant and $\SL(n, \RR)$-equi\-variant valuation if and only if there exist
		\begin{itemize}
			\item $c \in \RR$;
			\item $\nu\in\MeasC^+(\RR)$ with $\int_{\RR\setminus\{0\}}|s|^{-1} d\nu(s) < \infty$ and $\int_{\RR\setminus\{0\}}s^{-1} d\nu(s) = 0$;
		\end{itemize}
		such that
		\begin{align}\label{eq:thmSLREquivVal}
			\Psi(f)[x] = c + \int_{\RR\setminus\{0\}} \frac{f(sx)-f(0)}{|s|^2} d\nu(s), \quad x \in \RR^n,
		\end{align}
		for every $f \in \Conv(\RR^n,\RR)$.
	\end{MainTheorem}
	Note that by setting $c=0$ and $\nu = \delta_{-1} + \delta_1$, we obtain the map $\Psi(f)[x] =f(x) + f(-x) - 2f(0)$, $x \in \RR^n$, a dually epi-translation-invariant version of the map introduced and investigated as difference function map in \cite{Colesanti2006}. In the proof of Theorem~\ref{mthm:charSLRequivVal}, we make use of a construction introduced in \cite{Knoerr2020a}, which allows us to associate a suitable distribution to any homogeneous valuation in this class, called Goodey--Weil distribution. The equivariance property then translates to an invariance property of the Goodey--Weil distribution, and we show that this implies that its support must be lower dimensional. The main technical result of this article establishes that valuations with this property can be obtained by a simple restriction procedure from valuations defined on the corresponding lower dimensional subspace (compare Theorem \ref{thm:restrSuppGWRestrVal}). In the setting of Theorem \ref{mthm:charSLRequivVal}, this allows us to reduce the problem to valuations which are homogeneous of degree $0$ or $1$ and thus to apply a previous result from \cite{Hofstaetter2023} for additive endomorphisms.
	
	In our second main result, we apply a similar strategy to characterize all $\SL(n,\RR)$-contravariant valuations, where a map $\Psi:\Conv(\RR^n,\RR)\rightarrow\Conv(\RR^n,\RR)$ is called $\SL(n,\RR)$-contra\-variant if $\Psi(f\circ g)=\Psi(f)\circ g^{-T}$ for $g\in \SL(n,\RR)$ and $f\in\Conv(\RR^n,\RR)$. Except for the special case $n=2$, which can be deduced from Theorem~\ref{mthm:charSLRequivVal} by composing with a rotation by $\pi/2$, only trivial (that is, constant) valuations can appear here. Consequently, there is no direct analogue of the projection body map in this setting.
	\begin{MainTheorem}
		\label{mthm:charSLRcontrVal}
		A map $\Psi: \Conv(\RR^n, \RR) \rightarrow \Conv(\RR^n,\RR)$ is a continuous, dually epi-translation-invariant and $\SL(n, \RR)$-contra\-variant valuation if and only if
		\begin{itemize}
			\item $n=2$, and there exist $c \in \RR$ and $\nu\in\MeasC^+(\RR)$ with $\int_{\RR\setminus\{0\}}|s|^{-1} d\nu(s) < \infty$ and $\int_{\RR\setminus\{0\}}s^{-1} d\nu(s) = 0$, such that
			\begin{align}\label{eq:thmSLRcontrVal}
				\Psi(f)[x] = c + \int_{\RR\setminus\{0\}} \frac{f(s\vartheta x)-f(0)}{|s|^2} d\nu(s), \quad x \in \RR^2,
			\end{align}
			for every $f \in \Conv(\RR^2,\RR)$, where $\vartheta$ denotes a rotation by $\pi/2$;
			\item $n\geq 3$, and $\Psi(f) \equiv c$ for some $c \in \RR$ and every $f \in \Conv(\RR^n,\RR)$.
		\end{itemize}
	\end{MainTheorem}
	Let us point out that our proof of Theorem~\ref{mthm:charSLRcontrVal} applies in all settings where the notion of Goodey--Weil distribution is available and where these distributions have compact support, and that the same restrictions on the dimensions of the support apply in these cases.
	It is an interesting question whether one can modify the setting such that an appropriate analogue of the projection body map exists. There are four obvious candidates for changes to the conditions in Theorem~\ref{mthm:charSLRcontrVal}: change the invariance assumptions, change the topology, restrict the domain, or change the codomain. Some of the possible changes to the domain are covered by the extension results in \cite{Knoerr2024}, which show that one can in general not expect a qualitative change of the results.	Our next result shows that the situation also does not change if we change the codomain to the space $\Conv_{(0)}(\RR^n)$ of all convex functions $\RR^n \to (-\infty, \infty]$ which are lower semi-continuous and finite in a neighborhood of the origin. 
	
	\begin{MainTheorem}
		\label{mthm:stateConvO}
		Let $\Psi: \Conv(\RR^n, \RR) \to \ConvO(\RR^n)$ be a continuous, dually epi-translation-invariant valuation. If $\Psi(0)<\infty$, then $\Psi(\Conv(\RR^n, \RR)) \subset \Conv(\RR^n, \RR)$.
		
		Hence, Theorems~\ref{mthm:charSLRequivVal} and \ref{mthm:charSLRcontrVal} hold verbatim for valuations with values in $\ConvO(\RR^n)$.
	\end{MainTheorem}
Let us remark that the methods used in the proofs of Theorems~\ref{mthm:charSLRequivVal} and \ref{mthm:charSLRcontrVal} can not be applied directly to the setting in Theorem~\ref{mthm:stateConvO} since we can not simply consider $\Psi$ as a family of real-valued valuations. Instead, the proof relies on an extension result for valuations defined on dense subspaces of $\Conv(\RR^n,\RR)$ obtained in \cite{Knoerr2024}.

	\bigskip
	
	By a direct adaption of the methods, a similar statement as Theorem~\ref{mthm:charSLRcontrVal} also holds in complex vector spaces, where the real group $\SL(n,\RR)$ is replaced by $\SL(n,\CC)$.
	\begin{MainTheorem}
		\label{mthm:charSLCcontrVal}
		Suppose that $n\geq 3$. If $\Psi: \Conv(\CC^n, \RR) \rightarrow \Conv(\CC^n,\RR)$ is a continuous, dually epi-translation-invariant and $\SL(n, \CC)$-contra\-variant valuation, then $\Psi(f) \equiv c$ for some $c \in \RR$ and every $f \in \Conv(\CC^n,\RR)$.
	\end{MainTheorem}	
	Here, a map $\Psi:\Conv(\CC^n,\RR)\rightarrow\Conv(\CC^n,\RR)$ is called $\SL(n,\CC)$-contra\-variant if $\Psi(f\circ g)=\Psi(f)\circ g^{-*}$ for $f\in\Conv(\CC^n,\RR)$ and $g\in \SL(n,\CC)$, where $g^{-*}$ denotes the conjugate transpose of $g^{-1}$.

	Finally, we mention an immediate corollary of Theorem~\ref{mthm:charSLRcontrVal}.
	\begin{MainCorollary}
		\label{mthm:charSLRequivValReal}
		If $\Psi: \Conv(\RR^n, \RR) \rightarrow \RR^n$ is a continuous, dually epi-translation-invariant and $\SL(n, \RR)$-equi\-variant valuation, then $\Psi(f) \equiv 0$, $f \in \Conv(\RR^n,\RR)$.
	\end{MainCorollary}
	An analogous statement for $\CC^n$ can be deduced for $n \geq 3$ from Theorem~\ref{mthm:charSLCcontrVal}.

	\section{Background material}
	In this section we review the necessary background on valuations on convex functions. We will keep the exposition short and introduce only the notions needed in this article. For more details see, e.g., \cite{Hofstaetter2023, Knoerr2020a, Colesanti2019b, Colesanti2020, Ludwig2023} and the references therein.
	
	\medskip
	
	
	Let $F$ be a real Hausdorff topological vector space. We denote by $\VConv(\RR^n, F)$ the space of all continuous (with respect to epi-conver\-gence), dually epi-translation-invariant valuations on $\Conv(\RR^n, \RR)$ with values in $F$. The subspace of $k$-homoge\-neous valuations, that is, all $\mu \in \VConv(\RR^n, F)$ such that $\mu(\lambda f) = \lambda^k \mu(f)$ for all $f \in \Conv(\RR^n, \RR)$ and $\lambda > 0$, is denoted by $\VConv_k(\RR^n, F)$. If $F=\RR$, we will just write $\VConv(\RR^n)$ resp. $\VConv_k(\RR^n)$. It was recently proved in \cite{Colesanti2019b} for $F=\RR$ and later by different methods in \cite{Knoerr2020a} for general $F$ that $\VConv(\RR^n, F)$ admits a homogeneous decomposition.	
	\begin{theorem}[\cite{Knoerr2020a, Colesanti2019b}]
		\label{thm:mcmullenVConv}
		Suppose that $F$ is a real Hausdorff topological vector space. Then
		\begin{align*}
			\VConv(\RR^n, F) = \bigoplus_{k=0}^n \VConv_k(\RR^n, F).
		\end{align*}
	\end{theorem}
	
	Theorem~\ref{thm:mcmullenVConv} is equivalent to the (non-trivial) fact that for every $\mu \in \VConv(\RR^n, F)$ and $f_1, \dots, f_m \in \Conv(\RR^n, \RR)$, the map $(\lambda_1, \dots, \lambda_m) \to \mu(\lambda_1 f_1 + \dots + \lambda_m f_m)$ is a polynomial in $\lambda_i \geq 0$. For $k$-homogeneous $\mu$, the coefficients of this polynomial give rise to the \emph{polarization} $\bar{\mu}$ of $\mu$, also given by (see, e.g., \cite{Knoerr2020a}*{eq.~(1)})
	\begin{align*}
		\bar{\mu}(f_1,\dots,f_k)=\frac{1}{k!}\left.\frac{\partial^k}{\partial \lambda_1\dots\partial \lambda_k}\right|_0\mu\left(\sum_{j=1}^{k}\lambda_j f_j\right), \quad f_1, \dots, f_k \in \Conv(\RR^n, \RR).
	\end{align*}
	The map $\bar \mu$ is multi-linear in each component and thereby can be extended uniquely to differences of convex functions, and thus extends to a continuous multilinear functional on smooth function on $\RR^n$ with compact support. If $F$ is a locally convex vector space, that is, a Hausdorff topological vector space whose topology is generated by a family of semi-norms, this construction gives rise to a distribution on $(\RR^n)^k$, called the \emph{Goodey--Weil distribution} $\GW(\mu)$ of $\mu$. This notion is motivated by a similar construction due to Goodey and Weil \cite{Goodey1984}. In the following theorem, we list a few properties of the Goodey--Weil distributions established in \cite{Knoerr2020a}*{eq.~(3), Def.~5.3, Thm.~5.5, Thm.~5.7}. Here, we denote by $\overline{F}$ the completion of $F$ and by $\Distrib((\RR^n)^k, \overline{F})$ the space of distributions on $(\RR^n)^k$ with values in $\overline{F}$. We refer to \cite{Hoermander2003} for a background on distributions.
	\begin{theorem}[\cite{Knoerr2020a}]
		\label{thm:exGWdistrVConv}
		Let $F$ be a locally convex vector space that admits a continuous norm, and $1 \leq k \leq n-1$. For every $\mu \in \VConv_k(\RR^n, F)$, there exists a unique distribution $\GW(\mu) \in \Distrib((\RR^n)^k, \overline{F})$ with compact support and the following property: If $f_1, \dots, f_k \in \Conv(\RR^n, \RR) \cap C^\infty(\RR^n)$, then 
		\begin{align}\label{eq:defGWdistr}
			\GW(\mu)[f_1 \otimes \dots \otimes f_k] = \bar{\mu}(f_1, \dots, f_k).
		\end{align}
		Moreover, $\GW(\mu)$ has the following properties:
		\begin{enumerate}
			\item[(i)] If $k=1$, then $\GW(\mu)$ is of order at most $2$.
			\item[(ii)]\label{thm:suppGWDiag} The support of $\GW(\mu)$ is contained in the diagonal of $(\RR^n)^k$.
		\end{enumerate}
	\end{theorem} 
	The relation of $\GW(\mu)$ to $\mu$ in \eqref{eq:defGWdistr} and Theorem~\ref{thm:exGWdistrVConv}(ii) motivate the following definition.
	\begin{definition}[\cite{Knoerr2020a}]
		Let $\mu \in \VConv_k(\RR^n, F)$, $k\ge 1$. The support of $\mu$ is defined by $\supp \mu = \Delta^{-1}_k(\supp \GW(\mu))$. Here $\Delta_k:\RR^n\rightarrow(\RR^n)^k$, $x\mapsto (x,\dots,x)$, denotes the diagonal embedding. \\
		If $\mu\in\VConv(\RR^n,F)$ and $\mu=\sum_{k=0}^{n}\mu_k$ is the decomposition into its homogeneous components, then the support of $\mu$ is the set $\supp\mu=\bigcup_{k=1}^n\supp\mu_k$.
	\end{definition}
	It was shown in \cite[Prop.~6.3]{Knoerr2020a} that $\mu(f)$ only depends on the values of $f\in\Conv(\RR^n,\RR)$ on a neighborhood of $\supp \mu$.

	\medskip
	
	We now turn to valuations with values in $\Conv(\RR^n, \RR)$. Let $\Psi: \Conv(\RR^n, \RR) \to \Conv(\RR^n, \RR)$ be a continuous, dually epi-translation-invariant valuation. For $x \in \RR^n$, we define
	\begin{align*}
		\Psi_x(f) = \Psi(f)[x], \quad f \in \Conv(\RR^n,\RR).
	\end{align*}
	Clearly, $\Psi_x \in \VConv(\RR^n)$, and $\Psi_x$ is $k$-homogeneous if $\Psi$ is. Moreover, any equivariance of $\Psi$ directly translates to an invariance of $\Psi_x$, described in the following simple lemma.
	\begin{lemma}\label{lem:equivPropGWdistr}
		Suppose that $\Psi: \Conv(\RR^n,\RR) \to \Conv(\RR^n,\RR)$ is a continuous, dually epi-translation-invariant valuation, and let $G\subseteq \GL(n,\RR)$ be a subgroup.
		
		\begin{enumerate}
			\item If $\Psi$ is $G$-equivariant, then $\supp\Psi_x$ is invariant under $G_x$, the stabilizer of $x$ in $G$.
			\item If $\Psi$ is $G$-equivariant and $k$-homogeneous, then for every $x \in \RR^n$ and $\eta \in G$,
			\begin{align*}
				\GW(\Psi_{\eta(x)})[\varphi] = \left(\eta \cdot \GW(\Psi_x)\right)[\varphi] = \GW(\Psi_x)[\eta^{-1}\cdot\varphi], \quad \varphi \in C^\infty_c((\RR^n)^k).
			\end{align*}
			In particular, $\GW(\Psi_x)$ is invariant under $G_x$. Here, $G$ operates on $(\RR^n)^k$ by the diagonal action.
		\end{enumerate}
	\end{lemma}
	\begin{proof}
		The second claim follows directly from the relation of $\GW(\Psi_x)$ to the polarization of $\Psi_x$ in Theorem~\ref{thm:exGWdistrVConv}. The first claim follows by applying the second claim to the different homogeneous components of $\Psi_x$.
	\end{proof}
	We will use Lemma~\ref{lem:equivPropGWdistr} in the proof of the main results to show that the supports of $\Psi_x$ have to be lower-dimensional for $\SL(n,\RR)$-equi-/contravariant $\Psi$.

	\medskip

	Previously, a variant of Lemma~\ref{lem:equivPropGWdistr} was used in \cite{Hofstaetter2023} to characterize all additive, $\GL(n,\RR)$-equivariant endomorphisms of $\Conv(\RR^n, \RR)$. Here, an endomorphism $\Psi$ is called \emph{dually translation-invariant} if $\Psi(f + \ell) = \Psi(f)$, $f \in \Conv(\RR^n, \RR)$, for all linear maps $\ell:\RR^n\rightarrow\RR$.
	\begin{theorem}[\cite{Hofstaetter2023}]\label{mthm:CharGLEquivEndo}
		A map $\Psi: \Conv(\RR^n,\RR) \rightarrow \Conv(\RR^n,\RR)$ is continuous, additive and $\GL(n, \RR)$-equi\-variant if and only if there exists $c \in \RR$ and $\nu\in\MeasC^+(\RR)$ with $\int_{\RR\setminus\{0\}}|s|^{-1} d\nu(s) < \infty$ such that
		\begin{align}\label{eq:MainThmGLEquiv}
			\Psi(f)[x] = cf(0) + \int_{\RR\setminus\{0\}} \frac{f(sx)-f(0)}{|s|^2} d\nu(s), \quad x \in \RR^n,
		\end{align}
		for every $f \in \Conv(\RR^n,\RR)$. Moreover, the map $\Psi$ defined by \eqref{eq:MainThmGLEquiv} is dually translation-invariant if and only if $\int_{\RR\setminus\{0\}} s^{-1} d\nu(s) = 0$.
	\end{theorem}
	As the space of dually epi-translation-invariant additive endomorphisms of finite convex functions coincides with the space of $1$-homogeneous, dually epi-translation-invariant valuations, Theorem~\ref{mthm:CharGLEquivEndo} is an important building block in the proof of Theorem~\ref{mthm:charSLRequivVal}. Indeed, the main contribution of this article is to reduce the problem of characterizing equivariant valuations to the study of additive endomorphisms.

	\section{Preparations}
	In this section, we prove the main tool for the proofs of the main results. Informally speaking, we show that the dimension of the support of a valuation on $\Conv(\RR^n, \RR)$ yields a bound on the degrees of its homogeneous components.
	
	\begin{theorem}\label{thm:restrSuppGWRestrVal}
		Suppose that $0 \leq k \leq n$ and that $F$ is a locally convex vector space.
		
		If $\mu \in \VConv_k(\RR^n,F)$ satisfies $\supp \mu \subseteq E$ for some $E \in \Grass_i(\RR^n)$, $0 \leq i \leq n-1$, then there exists $\mu_E\in\VConv_k(E, F)$, such that
		\begin{align}\label{eq:RestrSuppStatement}
			\mu(f) = \mu_E(f|_E) 
		\end{align}
		for all $f \in \Conv(\RR^n, \RR)$. In particular, $\mu = 0$ whenever $k > i$.
	\end{theorem}
	\begin{proof}
		First, note that the claim is equivalent to
		\begin{align*}
			\mu(f) = \mu(\pi_E^\ast (f|_E)), \quad \text{for all}~f \in \Conv(\RR^n,\RR),
		\end{align*}
		where $\pi_E:\RR^n\rightarrow E$ is the orthogonal projection and $\pi_E^\ast g = g \circ \pi_E$, $g \in C(E)$, since we can then set $\mu_E(f):=\mu((\pi_E^*f)|_E)$ for $f\in\Conv(E,\RR)$. Since the topological dual $F'$ separates points in $F$, it is furthermore sufficient to establish this equation after composing $\mu$ with a continuous linear functional in $F'$. In particular, we may assume that $F=\RR$.
		
		Next, observe that we can restrict ourselves to the case $k=1$. Indeed, assume the case $k=1$ is proven and let $\mu \in \VConv_k(\RR^n)$ with $k>1$ be given such that $\supp \mu \subseteq E$ for some $E \in \Grass_i(\RR^n)$. Then its polarization $\bar{\mu}$ is a $1$-homogeneous, dually epi-translation-invariant valuation in each argument, which still has support in $E$. Consequently, we may apply the case $k=1$ to deduce
		\begin{align*}
			\bar{\mu}(f_1,\dots,f_k)=\bar{\mu}(\pi_E^*(f_1|_E),\dots,\pi_E^*(f_k|_E))
		\end{align*}
		for all $f_1,\dots,f_k\in\Conv(\RR^n,\RR)$. In particular,
		\begin{align*}
			\mu(f)=\bar{\mu}(f,\dots,f)=\bar{\mu}(\pi_E^*(f|_E),\dots,\pi_E^*(f|_E))=\mu(\pi_E^*(f|_E)),\quad\forall f\in\Conv(\RR^n,\RR).
		\end{align*}
		Moreover, using a direct induction argument, we can assume that $i=\dim E=n-1$.
%
%
		
		\medskip
		
		We are thus left to prove the claim for $k=1$ and $i=n-1$. Assuming w.l.o.g. that $E = \RR^{n-1} \times \{0\}$ and writing $x=(x',x_n)$ for $x \in \RR^n$, we have that $\supp \GW(\mu) = \supp \mu \subseteq \{x_n=0\}$, and, by \cite{Hoermander2003}*{Thm.~2.3.5}, there exist distributions $u_\alpha \in \DistribC(\RR^{n-1})$ of order at most $2-\alpha$, $\alpha = 0, 1, 2$, having compact support, such that
		\begin{align}\label{eq:prfRestrValToRkGW1}
			\GW(\mu)[\varphi] = u_0(\varphi(\cdot, 0)) + u_1((\partial_{x_n} \varphi)(\cdot, 0)) + u_2((\partial_{x_n}^2 \varphi)(\cdot, 0)),
		\end{align}
		for all $\varphi \in C^\infty_c(\RR^n)$. Here, we used that the order of $\GW(\mu)$ is at most $2$, compare Theorem~\ref{thm:exGWdistrVConv}. We need to show that $u_1 \equiv 0 \equiv u_2$.
		
		Let $\varphi' \in C^\infty_c(\RR^{n-1})$ be arbitrary and consider for $\varepsilon > 0$ the function
		\begin{align*}
			\psi_\varepsilon(x',x_n) = \varphi'(x')\sqrt{x_n^2 + \varepsilon^2}, \quad (x',x_n) \in \RR^n.
		\end{align*}
		Note that
		\begin{align*}
			\left.\partial_{x_n}^2 \psi_\varepsilon(x',x_n)\right|_{x_n = 0} = \varphi'(x')\left( -\frac{x_n^2}{\sqrt{x_n^2 + \varepsilon^2}^3} + \frac{1}{\sqrt{x_n^2 + \varepsilon^2}} \right)_{x_n = 0}  = \varphi'(x')\frac{1}{\varepsilon}
		\end{align*}
		and that the Hessian of $\psi_\varepsilon$ is given by
		\begin{align*}
			D^2 \psi_\varepsilon = (D^2 \varphi') \sqrt{x_n^2 + \varepsilon^2} + \nabla \varphi' \otimes \nabla \sqrt{x_n^2 + \varepsilon^2} + \nabla \sqrt{x_n^2 + \varepsilon^2} \otimes \nabla \varphi' \\+ \varphi' \mathrm{diag}\left(0, \dots, 0, -\frac{x_n^2}{\sqrt{x_n^2 + \varepsilon^2}^3} + \frac{1}{\sqrt{x_n^2 + \varepsilon^2}}\right).
		\end{align*}
		The norm of the first term is bounded on $\RR^{n-1}\times(-1,1)$ by $2\|\varphi'\|_{C^2_b(\RR^{n-1})}$ and grows slower than $2|x_n|\|\varphi'\|_{C^2_b(\RR^{n-1})}$ outside, for $\varepsilon$ small enough. The second and the third term can be bounded globally by $2\|\varphi'\|_{C^2_b(\RR^{n-1})}$, as $\sqrt{x_n^2 + \varepsilon^2}$ is Lipschitz-continuous with Lipschitz constant $1$.
		
		Now, let $C_1 = 4\|\varphi'\|_{C^2_b(\RR^{n-1})}$, $C_2 = \|\varphi'\|_{C^2_b(\RR^{n-1})}$ and $C_3 = \|\varphi'\|_\infty$, and, for $\varepsilon \geq 0$,
		\begin{align*}
			f_\varepsilon(x',x_n) = C_1 \frac{\|x\|^2}{2} + C_2 \|x\|^4 + C_3 \sqrt{x_n^2 + \varepsilon^2}, \quad (x',x_n) \in \RR^n.
		\end{align*}
		Then the Hessian of $f_\varepsilon + \psi_\varepsilon$ is positive semi-definite by the choices of $C_1, C_2$ and $C_3$, that is, $f_\varepsilon + \psi_\varepsilon$ is convex for $\varepsilon > 0$ small enough. Moreover, $f_\varepsilon + \psi_\varepsilon \to f_0 + \varphi' |x_n|$ uniformly on compact sets as $\varepsilon \to 0^+$, and hence $f_0 + \varphi' |x_n|$ is convex as well.
		Since $\psi_\varepsilon$ is smooth for $\varepsilon > 0$ and $\GW(\mu)$ has compact support, we can calculate
		\begin{align}\label{eq:prfRestrThmSuppEvalPsieps}
			\GW(\mu)[\psi_\varepsilon] = \mu(f_\varepsilon + \psi_\varepsilon) - \mu(f_\varepsilon), \quad \varepsilon > 0,
		\end{align}
		where the right-hand side converges to $\mu(f_0 + \varphi' |x_n|) - \mu(f_0)$ as $\varepsilon \to 0^+$. By \eqref{eq:prfRestrValToRkGW1}, the left-hand side of \eqref{eq:prfRestrThmSuppEvalPsieps} is given by
		\begin{align*}
			\GW(\mu)[\psi_\varepsilon] &= \varepsilon u_0(\varphi') + \left.\frac{x_n}{\sqrt{x_n^2 + \varepsilon^2}}\right|_{x_n=0} \!\!\!\!\!\!\!\!\!u_1(\varphi') + \left(-\frac{x_n^2}{\sqrt{x_n^2 + \varepsilon^2}^3} + \frac{1}{\sqrt{x_n^2 + \varepsilon^2}}\right)_{x_n=0} \!\!\!\!\!\!\!\!\!u_2(\varphi') \\
			&= \varepsilon u_0(\varphi') + \frac{1}{\varepsilon} u_2(\varphi'),
		\end{align*}
		which diverges for $\varepsilon \to 0$ if $u_2(\varphi') \neq 0$. We conclude that $u_2(\varphi') = 0$, and, as $\varphi'$ was arbitrary, that $u_2 \equiv 0$.
		
		Next, let $\varphi' \in C_c^\infty(\RR^{n-1})$ be arbitrary and consider for $\varepsilon > 0$ the functions
		\begin{align*}
			g_\varepsilon^\pm(x',x_n) = \varphi'(x') \sqrt{(x_n \pm \varepsilon)^2 + \varepsilon^2}, \quad (x',x_n) \in \RR^n. 
		\end{align*}
		Letting $C_4 = 4\|\varphi'\|_{C^2_b(\RR^{n-1})}$, $C_5 = \|\varphi'\|_{C^2_b(\RR^{n-1})}$ and $C_6 = \|\varphi'\|_\infty$, we define
		\begin{align*}
			h_\varepsilon^\pm = C_4 \frac{\|x\|^2}{2} + C_5\|x\|^4 + C_6 \sqrt{(x_n \pm \varepsilon)^2 + \varepsilon^2}.
		\end{align*}
		Then, by the same reasoning as before, for $\varepsilon$ small enough, the Hessian of $h_\varepsilon^\pm + g_\varepsilon^\pm$ is positive semi-definite and thus $h_\varepsilon^\pm + g_\varepsilon^\pm$ is convex. Moreover, as $\varepsilon \to 0^+$,
		\begin{align*}
			h_\varepsilon^\pm(x) \to h_0(x) := C_4\frac{\|x\|^2}{2} + C_5\|x\|^4 + C_6 |x_n|, \quad x \in \RR^n,
		\end{align*}
		and therefore $h_\varepsilon^\pm(x) + g_\varepsilon^\pm(x) \to h_0(x) + \varphi'(x') |x_n|$, which is therefore also convex.
		
		Consequently, since $g_\varepsilon^\pm$ is smooth for $\varepsilon > 0$, we can calculate
		\begin{align}\label{eq:prfRestrValToRkGW2}
			\GW(\mu)[g_\varepsilon^\pm] = \mu(h_\varepsilon^\pm + g_\varepsilon^\pm) - \mu(h_\varepsilon^\pm),
		\end{align}
		where the right-hand side converges for $\varepsilon \to 0^+$ to the same value $\mu(h_0 + \varphi' |x_n|) - \mu(h_0)$. However, by \eqref{eq:prfRestrValToRkGW1}, the left-hand side of \eqref{eq:prfRestrValToRkGW2} is given by
		\begin{align*}
			\GW(\mu)[g_\varepsilon^\pm] = \sqrt{2}\varepsilon u_0(\varphi') \pm \frac{1}{\sqrt{2}} u_1(\varphi') \to \pm \frac{1}{\sqrt{2}} u_1(\varphi'), \quad \varepsilon \to 0^+.
		\end{align*}
		We deduce that $u_1(\varphi') = 0$, that is, as $\varphi'$ was arbitrary, that $u_1\equiv 0$. Consequently, if $f\in\Conv(\RR^n,\RR)\cap C^\infty(\RR^n)$, then
		\begin{align*}
			\mu(f)=\GW(\mu)[f]=u_0(f(\cdot,0))=u_0(f|_E)=\GW(\mu)[\pi_E^*f|_E]=\mu(\pi_E^*f|_E).
		\end{align*} 
		This implies the first statement by approximation.
		
		Now assume that $k>i$. Then $\mu_E\in \VConv_k(E,F)$, where $\dim E=i<k$, and thus $\mu_E=0$ by the  McMullen-type composition in Theorem~\ref{thm:mcmullenVConv}. Consequently, $\mu=0$.
	\end{proof}

\begin{remark}
	Theorem \ref{thm:restrSuppGWRestrVal} has the following interpretation in terms of the Goodey--Weil distributions: If $\mu\in\VConv_k(\RR^n,F)$ is given by $\mu(f)=\mu_E(f|_E)$ for some $\mu_E\in \VConv_k(\RR^n,F)$, $E\in \Grass_i(\RR^n)$, then for every $\varphi \in C^\infty_c((\RR^n)^k)$
	\begin{align*}
		\GW(\mu)[\varphi] = \GW(\mu_E)[\varphi|_{E^k}].
	\end{align*}
	In other words, $\GW(\mu)[\varphi]$ only depends on the derivatives of $\varphi$ tangential to $E^k$.
\end{remark}

\begin{remark}
		Let us point out that in the proof of Theorem~\ref{thm:restrSuppGWRestrVal} we only work with the Goodey--Weil distribution of $\mu$. Consequently, the proof directly extends to the much larger class of valuations that admit Goodey--Weil distributions (see, e.g., \cite{Knoerr2024}). However, we will not need this generality in the following.
\end{remark}

We mention a direct consequence of Theorems~\ref{thm:restrSuppGWRestrVal} and \ref{thm:mcmullenVConv} for later reference.

\begin{corollary}\label{cor:vconvRestrSuppNonHom}
	Suppose that $\mu \in \VConv(\RR^n)$ satisfies $\supp \mu \subseteq E$ for some $E \in \Grass_i(\RR^n)$, $0 \leq i \leq n-1$. Then there exists $\mu_E \in \VConv(E)$ such that
	\begin{align*}
		\mu(f) = \mu_E(f|_E), \quad \forall f \in \Conv(\RR^n, \RR).
	\end{align*}
	Hence, $\mu(f)$ depends only on the values of $f|_E$, and $\mu \in \bigoplus_{k=0}^{i} \VConv_k(\RR^n)$.
\end{corollary}

%
	\section{Proof of the main results}
	In this section, we prove the main results for valuations on finite convex functions. The proof of Theorem~\ref{mthm:stateConvO} for valuations with values in $\ConvO(\RR^n)$ is postponed to Section~\ref{sec:convOStatements}.
	
	\begin{proof}[Proof of Theorem~\ref{mthm:charSLRequivVal}]
		Note that  \eqref{eq:thmSLREquivVal} defines a continuous, dually epi-translation-in\-variant and $\SL(n, \RR)$-equivariant valuation by Theorem~\ref{mthm:CharGLEquivEndo}. It remains to see that every such map has the desired representation.
		
		Let $\Psi: \Conv(\RR^n, \RR) \rightarrow \Conv(\RR^n,\RR)$ be a continuous, dually epi-translation-invariant and $\SL(n, \RR)$-equivariant valuation, and fix $x \in \RR^n$. By Theorem~\ref{thm:exGWdistrVConv}, $\Psi_x$ has compact support and, by Lemma~\ref{lem:equivPropGWdistr}, it is invariant under $\SL(n,\RR)_x$. If $x = 0$, this implies that $\supp \Psi_x \subseteq \{0\}$, as $\SL(n,\RR)$ acts transitively on $\RR^n \setminus\{0\}$.
		
		If $x \neq 0$, then $\SL(n,\RR)_x$ contains all shear maps fixing $x$. In particular, if $y \not \in \linspan\{x\}$, then the orbit of $y$ under $\SL(n,\RR)_x$ contains the line $y + tx, t \in \RR$, which is not compact. Consequently, $y \not \in \supp\Psi_x$, that is, $\supp \Psi_x \subset \linspan\{x\}$.
		
		Hence, by Corollary~\ref{cor:vconvRestrSuppNonHom} applied for $E = \linspan\{x\}$, there exist valuations $\Psi_x^k \in \VConv_k(\RR^n)$, $k=0,1$, such that $\Psi_x = \Psi_x^0 + \Psi_x^1$, $x \in \RR^n$. Clearly, by homogeneity and continuity, $\Psi_x^0(f) = \Psi_x^0(0) = \Psi_x(0)$, $x \in \RR^n$ and $f \in \Conv(\RR^n,\RR)$. Since $\Psi(0)$ is a $\SL(n,\RR)$-invariant and continuous function, $\Psi(0)$ is constant and $\Psi^0(f) \equiv c$, for some $c \in \RR$.
		
		Consequently, setting $\Psi^1(f) = \Psi(f) - c$, $f \in \Conv(\RR^n, \RR)$, we obtain a dually epi-translation-invariant, $1$-homogeneous valuation $\Psi^1: \Conv(\RR^n, \RR) \to \Conv(\RR^n,\RR)$, which clearly is also continuous and $\SL(n,\RR)$-equivariant.
		
		 We claim that $\Psi^1$ is in fact $\GL(n,\RR)$-equivariant. Indeed, repeating the argument, we see that $f\mapsto\Psi^1(f)[x]$ is supported on $\linspan\{x\}$, and consequently Corollary~\ref{cor:vconvRestrSuppNonHom} shows that $\Psi^1(f)[x]$ only depends on the values of $f\in\Conv(\RR^n, \RR)$ on $\linspan\{x\}$, $x \in \RR^n$. If $g\in\GL(n,\RR)$  is arbitrary and $x\in \RR^n$, we choose $\eta\in \SL(n,\RR)$ with $\eta x=gx$, which is possible since $n\ge 2$. Then $f\circ g=f\circ \eta$ on $\linspan\{x\}$. Since $\Psi^1(\cdot)[x]$ only depends on the values of its argument on $\linspan\{x\}$, we obtain
	 	\begin{align*}
	 		\Psi^1(f\circ g)[x]=\Psi^1(f\circ \eta)[x]=\Psi^1(f)[\eta x]=\Psi^1(f)[gx].
	 	\end{align*}
 		Thus $\Psi^1$ is $\GL(n,\RR)$-equivariant.	As a dually epi-translation-invariant, $1$-homoge\-neous valuation, $\Psi^1$ is also additive (this follows directly from Theorem~\ref{thm:mcmullenVConv} and the comment below it). We can therefore apply Theorem~\ref{mthm:CharGLEquivEndo} to $\Psi^1$,  which shows the representation~\eqref{eq:thmSLREquivVal} for $\Psi$.
	\end{proof}

	\begin{remark}
		Note that in the proof of Theorem~\ref{mthm:charSLRequivVal}, convexity of $\Psi(f)$ is only needed in the last step for the application of Theorem~\ref{mthm:CharGLEquivEndo}. In particular, the restriction on the possible degrees of homogeneity also applies to valuations with values in different function spaces, for example continuous functions, integrable functions, etc.
	\end{remark}
	
	\begin{proof}[Proof of Theorem~\ref{mthm:charSLRcontrVal}]
		Let $\Psi: \Conv(\RR^n, \RR) \rightarrow \Conv(\RR^n,\RR)$ be a continuous, dually epi-translation-invariant and $\SL(n, \RR)$-contravariant valuation, and fix $x \in \RR^n$. By Theorem~\ref{thm:exGWdistrVConv}, $\Psi_x$ has compact support and, by Lemma~\ref{lem:equivPropGWdistr}, it is invariant under the group
		\begin{align*}
			G_x = \{\eta \in \SL(n,\RR): \eta^T x = x \}.
		\end{align*}
		If $x = 0$, we conclude that $\supp \Psi_x \subseteq \{0\}$. If $x \neq 0$, then $G_x$ contains all maps $\eta \in \SL(n,\RR)$, such that $\eta^T$ is a shear map fixing $x$. Letting $y \in \RR^n$ and $z \in \RR^n \setminus \linspan\{x\}$, this implies that
		\begin{align*}
			\{\langle \eta y, z\rangle: \, \eta \in G_x\}  \supseteq \{\langle y, z\rangle + t \langle y,x \rangle: \, t \in \RR\}.
		\end{align*}
		If $y \in \supp \Psi_x$, then the set on the left-hand side must be bounded as it is contained in the image of $\supp\Psi_x$ under the continuous map $u \mapsto \langle u, z\rangle$. Hence, $\langle y, x \rangle$ must be zero, that is, $\supp \Psi_x \subseteq x^\perp$.
		
		Next, assume that $n \geq 3$. Then the group $G_x$ contains the subgroup $H_x \cong \SL(n-1,\RR)$ consisting of all maps $\eta \in \SL(n,\RR)$ such that $\eta x = x = \eta^T x$. As $H_x$ acts transitively on $x^\perp \setminus \{0\}$, the compactness of $\supp\Psi_x$ implies that $\supp\Psi_x \subseteq \{0\}$.
		
		Consequently, by Corollary~\ref{cor:vconvRestrSuppNonHom}, $\Psi_x$ must be $0$-homogeneous for every $x \in \RR^n$. By homogeneity and continuity, $\Psi_x(f) = \Psi_x(0)$, $f \in \Conv(\RR^n, \RR)$, and by $\SL(n,\RR)$-contravariance, $\Psi(0)\equiv c$ for some $c \in \RR$. Hence, $\Psi(f) \equiv c$, as claimed. 
		
		It remains to consider the case $n=2$. Here, we set $\widetilde{\Psi}(f) = \Psi(\vartheta^{-1} f)$, $f \in \Conv(\RR^2,\RR)$, where we denote by $\vartheta$ a rotation by $\pi/2$. Since $\eta\vartheta = \vartheta \eta^{-T}$ for any $\eta \in \SL(2,\RR)$, $\widetilde{\Psi}^1$ is $\SL(2,\RR)$-equivariant. The claim for $n=2$ thus follows from Theorem~\ref{mthm:charSLRequivVal}.
	\end{proof}

	\begin{proof}[Proof of Theorem~\ref{mthm:charSLCcontrVal}]
		The proof of Theorem~\ref{mthm:charSLCcontrVal} follows exactly with the same arguments as the proof of Theorem~\ref{mthm:charSLRcontrVal}. In this case, the support of $\Psi_z$, $z\in\CC^n$, has to be invariant under the group
		\begin{align*}
			G_z=\{\eta\in \SL(n,\CC):\eta^*z=z\},
		\end{align*}
		where $\eta^*$ denotes the conjugate transpose of $\eta$. For $n\ge3$, this group contains all $\eta$ such that $\eta^*$ is a complex shear map fixing $z$, as well as the the subgroup of all $\eta\in \SL(n,\CC)$  that satisfy $\eta z=z=\eta^*z$, which is isomorphic to $\SL(n-1,\CC)$ for $z\ne 0$. Now the same arguments as in the proof of Theorem~\ref{mthm:charSLRcontrVal} show that $\supp \Psi_z\subset \{0\}$ for all $ z\in \CC^n$. In particular, $\Psi_z$ is $0$-homogeneous for every $z\in\CC^n$ by Corollary~\ref{cor:vconvRestrSuppNonHom}, which implies
		\begin{align*}
			\Psi(f)[z]=\Psi_z(f)=\Psi_z(0)=\Psi(0)[z].
		\end{align*}
		Since $\Psi(0)\in\Conv(\CC^n,\RR)$ is an $\SL(n,\CC)$-invariant function, it has to be constant, which completes the proof.
	\end{proof}

 	\begin{proof}[Proof of Corollary~\ref{mthm:charSLRequivValReal}]
 		Let $\Psi: \Conv(\RR^n, \RR) \rightarrow \RR^n$ be a continuous, dually epi-translation-invariant and $\SL(n, \RR)$-equivariant valuation. Noting that the canonical map $\RR^n \to (\RR^n)^\ast \subset \Conv(\RR^n, \RR)$ is $\SL(n,\RR)$-contravariant, we obtain an $\SL(n,\RR)$-contravariant valuation $\widehat{\Psi}: \Conv(\RR^n, \RR) \to \Conv(\RR^n, \RR)$, defined by
 		\begin{align*}
 			\widehat{\Psi}(f)[x] = \langle x, \Psi(f) \rangle, \quad x \in \RR^n, f\in\Conv(\RR^n, \RR).
 		\end{align*}
 	
 		If $n \geq 3$, by Theorem~\ref{mthm:charSLRcontrVal}, $\widehat{\Psi}(f) \equiv c$, $f \in \Conv(\RR^n,\RR)$, for some $c \in \RR$. As $\widehat{\Psi}(f)$ is linear by definition, $c=0$ and the claim follows in this case.
 		
		If $n=2$, by Theorem~\ref{mthm:charSLRcontrVal}, eq.~\eqref{eq:thmSLRcontrVal}, there exist $c \in \RR$ and a measure $\nu \in \MeasC^+(\RR)$ such that
		\begin{align*}
			\widehat{\Psi}(f)[x] = c + \int_{\RR\setminus\{0\}} \frac{f(s\vartheta x) - f(0)}{|s|^2} d\nu(s), \quad x \in \RR^2,
		\end{align*}
		for all $f \in \Conv(\RR^2, \RR)$. Again, as $\widehat{\Psi}(f)$ is linear for every $f$, we conclude that $c=\widehat{\Psi}(f)[0] = 0$. Moreover, for $f \in \Conv(\RR^2,\RR)$ and $x,y \in \RR^2$, by convexity of $f$
		\begin{align*}
			\frac{1}{2} \widehat{\Psi}(f)[x] + \frac{1}{2} \widehat{\Psi}(f)[y] &= \widehat{\Psi}(f)\left[\frac{x+y}{2}\right] = \int_{\RR\setminus\{0\}}\frac{f\left(s\vartheta \frac{x+y}{2}\right) - f(0)}{|s|^2} d\nu(s)\\
			&\leq \int_{\RR\setminus\{0\}}\frac{f(s\vartheta x) - f(0)}{2|s|^2} d\nu(s) + \int_{\RR\setminus\{0\}}\frac{f(s\vartheta y) - f(0)}{2|s|^2} d\nu(s) \\&= \frac{1}{2} \widehat{\Psi}(f)[x] + \frac{1}{2} \widehat{\Psi}(f)[y].
		\end{align*}
		Since $\nu$ is a non-negative measure, this implies $f\left(s\vartheta \frac{x+y}{2}\right) = \frac{1}{2}f(s\vartheta x) + \frac{1}{2}f(s\vartheta y)$ for $\nu$-almost all $s \in \RR\setminus\{0\}$. Picking $f$ to be strictly convex, we deduce that $\nu = 0$, yielding the claim for $n=2$.
 	\end{proof}

	\begin{remark}
		Let us point out that our proof of Theorem~\ref{mthm:charSLRequivVal} can be directly adapted to the dually epi-translation-\emph{equivariant} setting. Indeed, let $\Psi: \Conv(\RR^n, \RR) \to \Conv(\RR^n, \RR)$ be dually epi-translation-equivariant, that is, $\Psi(f + \lambda) = \Psi(f) + \lambda$ for every affine $\lambda$. Then the map $\tilde \Psi = \Psi - \id$ is dually epi-translation-invariant, but will in general not map into $\Conv(\RR^n,\RR)$ anymore. As our argument for the Goodey--Weil distributions does not make use of the convexity at all, we can conclude similarly that if $\Psi$ is $\SL(n,\RR)$-equivariant, then $\tilde\Psi$ is the sum of a $0$- and a $1$-homogeneous valuation. Hence, $\Psi$ is the sum of these plus the identity. Subtracting the (constant) $0$-homogeneous component from $\Psi$, we obtain a map $\Psi^1$ satisfying the conditions of Theorem~\ref{mthm:CharGLEquivEndo}, which does not need dual epi-translation-invariance.
		
		Further results in this direction where obtained by Li in \cite{Li2023}.
	\end{remark}

	\section{Extension to $\ConvO(\RR^n)$ -- proof of Theorem~\ref{mthm:stateConvO}}
	\label{sec:convOStatements}
	In this section, we prove Theorem~\ref{mthm:stateConvO}. The proof is based on a dissection argument for valuations $\mathcal{K}^m \to \ConvO(\RR^n)$, which is then lifted to valuations on convex functions. In order to state the result for valuations on convex bodies, let $\dom f = \{x \in \RR^n: f(x) < \infty\}$ be the \emph{domain} of $f \in \ConvO(\RR^n)$.	
	
	\begin{lemma}\label{lem:domFinKConv0}
		Let $\Psi:\mathcal{K}^m\rightarrow \Conv_{(0)}(\RR^n)$ be a continuous valuation,
		\begin{align*}
			U:=\interior\left(\bigcap_{x\in\RR^m}\dom \Psi(\{x\})\right).
		\end{align*} Then $U\subset \dom \Psi(K)$ for all $K\in\mathcal{K}^m$.
	\end{lemma}
	\begin{proof}
		For $x\in\RR^m$ set $\mathcal{B}_{x,R}:=\{K\in\mathcal{K}^m: K\subset B^\infty_R(x)\}$, where $B^\infty_R(x)$ denotes the $L^\infty$-ball/cube with center $x$ and radius $R> 0$. Then $\mathcal{B}_{x,R}$ is compact in $\mathcal{K}^m$ by Blaschke's selection theorem. Let $A\subset U$ be a compact subset. We claim that there exists $\varepsilon_{x,A}>0$ such that 
		\begin{align}
			\label{eq:prfdomFKDefEpsxA}
			A\subset \dom\Psi(K) \quad\text{for all}~K\in\mathcal{B}_{x,\varepsilon_{x,A}}.
		\end{align}
		Otherwise, there exists a sequence $(K_j)_j$ in $\mathcal{K}^m$ converging to $\{x\}$ such that for every $j\in\mathbb{N}$ there exists a point $x_j\in A$ with $x_j\notin \dom \Psi(K_j)$. However, the sequence $\Psi(K_j)$ epi-converges to $\Psi(\{x\})$, which is finite on the compact set $A$. Since $A$ does not intersect the boundary of $\dom \Psi(\{x\})$, this convergence is uniform on $A$ (see, e.g., \cite{Rockafellar1998}*{Thm.~7.17}). In particular, there exists $N\in\mathbb{N}$ such that $\Psi(K_j)$ is finite on $A$ for all $j\ge N$, which is a contradiction.
		
		Now let $K\in\mathcal{K}^m$ be an arbitrary convex body. We will establish by induction on the dimension $k$ of $K$ that $A\subset \dom \Psi(K)$. For $k=0$ this follows directly from the definition of $U$. Now assume that we have shown the claim for all convex bodies of dimension at most $k-1$. For a given $K\in\mathcal{K}^m$ of dimension $k$, choose $R>0$ such that $K\subset B_R^\infty(0)$. We can then find a finite number of points $x_1,\dots, x_N\in \RR^m$ such that $B_{R}^\infty(0)$ is covered by the sets $B_{\varepsilon_{x_j,A}/2}^\infty(x_j)$, $j=1, \dots, N$. Set
		\begin{align*}
			\varepsilon:=\frac{1}{2}\min_{1\le j\le N}\varepsilon_{x_j,A},\quad \text{ and } \quad	a_i:=\min\{x_i: x\in K\}, \quad i=1, \dots, m,
		\end{align*}
		and consider for $l\in\mathbb{N}$ the hyperplanes
		\begin{align*}
			H_{i,l}:=\{x\in\RR^m: x_i=a_i+l\varepsilon\}.
		\end{align*}
		Similarly, denote by $H^\pm _{i,l}$ the positive and negative half spaces corresponding to these hyperplanes.
		For $1\le i\le m$ consider the sets
		\begin{align*}
			S_{i,l}:=H_{i,l}^+\cap H_{i,l+1}^-=\{x\in \RR^m: a_i+l\varepsilon\le x_i\le a_i+(l+1)\varepsilon\}.
		\end{align*}
		Then there exists a minimal $M>0$ such that
		\begin{align*}
			K=\bigcup_{l=0}^M (K\cap S_{i,l}).
		\end{align*}
		If $M>1$, then $K$ is not contained in a hyperplane parallel to the $i$th coordinate hyperplane, and in this case the pairwise intersection of two convex bodies in the decomposition on the right-hand side is either empty or $K\cap H_{i,l}$, that is, of dimension at most $k-1$. In particular, $\Psi(K\cap H_{i,l})$ is finite on $A$ for $1\le l\le M-1$ by the induction assumption. Applying the valuation property repeatedly, we thus obtain
		\begin{align*}
			\Psi(K)[x]+\sum_{l=1}^{M-1}\Psi(K\cap H_{i,l})[x]=\sum_{l=0}^M\Psi(K\cap S_{i,l})[x]\quad\text{for}~x\in A.
		\end{align*}
		Since $\Psi(K\cap H_{i,l})$ is finite on $A$ for $1\le l\le M-1$, this implies that $\Psi(K)$ is finite on $A$ if and only if all of the functions $\Psi(K\cap S_{i,l})$ are finite on $A$. Iterating this argument for all coordinates $1\le i\le m$, we see that this is the case if and only if the functions
		\begin{align*}
			\Psi(K\cap S_{1,l_1}\cap \dots\cap S_{m,l_m}),\quad l_1,\dots,l_m\in\mathbb{N}
		\end{align*}
		are finite on $A$ whenever the intersection is non-empty. By construction, these sets are contained in a cube $B_{\varepsilon}^\infty(x)$ for some $x\in B_R^\infty(0)$. Since any $x\in B_R^\infty$ is contained in $B^\infty_{\varepsilon_{x_j,A}/2}(x_j)$ for some $x_j$, $1\le j\le N$, and $\varepsilon\leq\frac{\varepsilon_{x_j,A}}{2}$, we see that 
		\begin{align*}
			K\cap S_{1,l_1}\cap \dots\cap S_{m,l_m}\subset B^\infty_{\varepsilon_{x_j,A}}(x_j).
		\end{align*}
		If this intersection is non-empty, this implies $K\cap S_{1,l_1}\cap \dots\cap S_{m,l_m}\in \mathcal{B}_{x_j,\varepsilon_{x_j,A}}$ and thus, by \eqref{eq:prfdomFKDefEpsxA}, $\Psi(K\cap S_{1,l_1}\cap \dots\cap S_{m,l_m})$ is finite on $A$, completing the induction step.
		
		In total we obtain that for every $K\in\mathcal{K}^m$, the function $\Psi(K)$ is finite on every compact subset of $U$. Thus $\Psi(K)$ is finite on $U$, so $U\subset \dom \Psi(K)$.
	\end{proof}

	\medskip
	
	Next, we want to lift Lemma~\ref{lem:domFinKConv0} to valuations on convex functions. We will use a connection between convex functions in $\RR^n$ and convex bodies in $\RR^{n+1}$ established in \cite{Knoerr2023}, as well as an extension result from \cite{Knoerr2024}.
	
	For $K \in \mathcal{K}^{n+1}$, define the function $\lfloor K \rfloor: \RR^n\to(-\infty,\infty]$ by (see \cite{Knoerr2023})
	\begin{align*}
		\lfloor K \rfloor(x) = \inf\{t \in \RR: \, (x,t) \in K\}, \quad x \in \RR^n,
	\end{align*}
	so in particular, $\lfloor K \rfloor(x) = \infty$, if $K \cap \{x\}\times \RR = \emptyset$.
	It is easily seen that $\lfloor K \rfloor \in \Conv_{\mathrm{cd}}(\RR^n)$, the space of all convex, lower semi-continuous functions $f: \RR^n\to (-\infty, \infty]$ such that $\dom f$ is non-empty and compact.
	Moreover, it was shown in \cite{Knoerr2023} that $\lfloor \cdot \rfloor: \mathcal{K}^{n+1} \to \Conv_{\mathrm{cd}}(\RR^n)$ is continuous and onto.
	
	Note that the space $\Conv_{\mathrm{cd}}(\RR^n)$ can be identified with a subset of $\Conv(\RR^n, \RR)$ by taking the convex conjugate (Legendre transform), $f^\ast (x) = \sup_{y \in \RR^n} \langle x, y \rangle - f(y)$. Indeed, the space $\Conv(\RR^n, \RR)$ is mapped by the Legendre transform to the space 
	\begin{align*} 
		\Conv_{\mathrm{sc}}(\RR^n) = \{f: \RR^n \to (-\infty,\infty]:\, &\text{$f$ is convex, lower semi-continuous},\\ &f \not \equiv +\infty, \text{ and } \lim\limits_{\|x\|\to \infty}\frac{f(x)}{\|x\|} = \infty\},
	\end{align*}
	of super-coercive convex functions. Clearly, $\Conv_{\mathrm{cd}}(\RR^n) \subset \Conv_{\mathrm{sc}}(\RR^n)$ and, thus, its image under the the Legendre transform satisfies $\Conv_{\mathrm{cd}}(\RR^n)^\ast \subset \Conv(\RR^n, \RR)$.	This subset is dense (\cite{Knoerr2023}*{Cor.~2.3}) and the following extension property holds.
	\begin{theorem}[\cite{Knoerr2024}]\label{thm:extValCdToSc}
		Suppose that $\mu: \Conv_{\mathrm{cd}}(\RR^n)^\ast \to \RR$ is a continuous, dually epi-translation-invariant valuation. Then $\mu$ extends uniquely by continuity to a continuous, dually epi-translation-invariant valuation $\hat{\mu}: \Conv(\RR^n, \RR) \to \RR$.
	\end{theorem}

	Let us further point out that for an affine function $\lambda = \langle v, \cdot \rangle + c$, its convex conjugate is given by $\lambda^\ast(x) = \mathbbm{1}_{\{v\}}^\infty - c = \lfloor\{(v,-c)\}\rfloor$, where $\mathbbm{1}^\infty_K$ denotes the convex indicator function of $K$ taking the value $0$ in $K$ and $\infty$ outside.
	
	We are now ready to prove Theorem~\ref{mthm:stateConvO}.
	\begin{proof}[Proof of Theorem~\ref{mthm:stateConvO}]
		Suppose that $\Psi: \Conv(\RR^n, \RR) \to \ConvO(\RR^n)$ is a continuous and dually epi-translation-invariant valuation such that $\Psi(0) \in \Conv(\RR^n, \RR)$.
		
		First, consider $\tilde{\Psi}:\mathcal{K}^{n+1}\rightarrow \Conv_{(0)}(\RR^n)$ given by $\tilde{\Psi}(K):=\Psi(\lfloor K \rfloor^\ast)$, $K \in \mathcal{K}^{n+1}$. Then, by the continuity of the Legendre transform and $\lfloor\cdot\rfloor$, $\tilde{\Psi}$ is a continuous map and as the Legendre transform exchanges minimum and maximum, it is also a valuation, compare \cite[Thm. 9]{Knoerr2023}. Consequently, Lemma~\ref{lem:domFinKConv0} implies
		\begin{align*}
			U\subset \dom \tilde{\Psi}(K)=\dom \Psi(\lfloor K\rfloor^\ast)\quad\text{for all}~K\in\mathcal{K}^{n+1},
		\end{align*}
		where $U$ is the largest open set contained in all $\dom \tilde{\Psi}(\{(x,t)\})$, $(x,t) \in \RR^{n+1}$. By the remark above the proof, $\lfloor (x,t) \rfloor = \lambda^\ast$ for $\lambda(y) = \langle x, y \rangle -t$. Hence, $\tilde{\Psi}(\{(x,t)\}) = \Psi(\lambda) = \Psi(0)$ for all $(x,t) \in \RR^{n+1}$, as $\Psi$ is dually epi-translation-invariant.
		
		We conclude that $U = \RR^n$. Since $\lfloor\cdot\rfloor:\mathcal{K}^{n+1}\rightarrow \Conv_{\mathrm{cd}}(\RR^n)$ is onto, this implies that $\Psi(f)$ is finite for all $f \in \Conv_{\mathrm{cd}}(\RR^n)^\ast$. By Theorem~\ref{thm:extValCdToSc}, $\Psi_x|_{\Conv_{\mathrm{cd}}(\RR^n)^\ast}$ extends uniquely by continuity to a real-valued (that is, finite) valuation on $\Conv(\RR^n, \RR)$ for every $x \in \RR^n$, which we denote by $\widehat{\Psi}_x$. For $f\in\Conv(\RR^n,\RR)$ define $\widehat{\Psi}(f):\RR^n\rightarrow\RR$ by $\widehat{\Psi}(f)[x]:=\widehat{\Psi}_x(f)$. Then $f\mapsto \widehat{\Psi}(f)$ is continuous with respect to pointwise convergence. Since this function is convex for functions belonging to the dense subspace $\Conv_{\mathrm{cd}}(\RR^n)^*$ of $\Conv(\RR^n,\RR)$ and pointwise limits of convex functions are again convex, $\widehat{\Psi}(f)\in\Conv(\RR^n,\RR)$ for every $f\in\Conv(\RR^n,\RR)$. By \cite[Thm.~7.58]{Rockafellar1998} the topology induced by epi-convergence coincides with the topology induced by pointwise convergence on $\Conv(\RR^n,\RR)$. In particular, we may consider this as a continuous map $\widehat{\Psi}:\Conv(\RR^n,\RR)\rightarrow\ConvO(\RR^n)$. Since this map coincides with $\Psi$ on the dense subspace $\Conv_{\mathrm{cd}}(\RR^n)^*$, these two maps have to coincide on $\Conv(\RR^n,\RR)$. In particular, $\Psi(f)\in\Conv(\RR^n,\RR)$ for every $f\in\Conv(\RR^n,\RR)$. This concludes the proof of the first statement.
		
		It remains to prove the second statement for equivariant valuations. To this end, let $\Psi: \Conv(\RR^n, \RR) \to \ConvO(\RR^n)$ be a continuous, dually epi-translation-invariant valuation which is $\SL(n,\RR)$-equi- resp.\ contravariant. By the equi- resp.\ contravariance, $\Psi(0)$ is $\SL(n,\RR)$-invariant. As $\Psi(0)\in \ConvO(\RR^n)$ is finite in a neighborhood of the origin, it must be constant and thus $\dom\Psi(0) = \RR^n$. Hence, the first part of the proof implies that $\Psi(\Conv(\RR^n, \RR)) \subset \Conv(\RR^n, \RR)$ and an application of Theorem~\ref{mthm:charSLRequivVal} resp.\ \ref{mthm:charSLRcontrVal} concludes the proof.
	\end{proof}

	Note that we actually proved the following statement.
	\begin{proposition}
		Let $\Psi: \Conv_{\mathrm{cd}}(\RR^n)^\ast \to \ConvO(\RR^n)$ be a continuous valuation. If $\Psi(h) \in \Conv(\RR^n, \RR)$ for every affine $h$, then $\Psi(\Conv_{\mathrm{cd}}(\RR^n)^\ast) \subset \Conv(\RR^n, \RR)$.
	\end{proposition}
	The condition of dual epi-translation-invariance in Theorem~\ref{mthm:stateConvO} was only needed in the last step to lift the statement to $\Conv(\RR^n, \RR)$.

	\section*{Acknowledgments}
	We would like to thank Fabian Mussnig for his continuing and long-lasting interest in this work.
	The first-named author was supported by the Austrian Science Fund (FWF), \href{https://doi.org/10.55776/P34446}{doi:10.55776/P34446}.

	\bibliographystyle{abbrv}
	
	\bibliography{./booksEquivVals}
	
%
%
%
	

\end{document}